\documentclass[12pt, reqno]{amsart}
\usepackage{amsmath, amsthm, amscd, amsfonts, amssymb, graphicx, color}
\usepackage[bookmarksnumbered, colorlinks, plainpages]{hyperref}
\hypersetup{colorlinks=true,linkcolor=red, anchorcolor=green, citecolor=cyan, urlcolor=red, filecolor=magenta, pdftoolbar=true}
\input{mathrsfs.sty}

\newtheorem{theorem}{Theorem}[section]
\newtheorem{lemma}[theorem]{Lemma}
\newtheorem{proposition}[theorem]{Proposition}

\theoremstyle{definition}
\newtheorem{definition}[theorem]{Definition}
\newtheorem{example}[theorem]{Example}

\theoremstyle{remark}
\newtheorem{remark}[theorem]{Remark}
\numberwithin{equation}{section}
\makeatletter
\newcommand*{\rom}[1]{\expandafter\@slowromancap\romannumeral #1@}
\makeatother
\begin{document}

\title[Power-norms based on Hilbert $C^*$-modules]{Power-norms based on Hilbert $C^*$-modules}

\author[S. Abedi, M.S. Moslehian]{Sajjad Abedi$^1$ \MakeLowercase{and} Mohammad Sal Moslehian$^2$}

\address{$^1$Department of Pure Mathematics, Ferdowsi University of Mashhad, P. O. Box 1159, Mashhad 91775, Iran}
\email{sajjadabedy1995@gmail.com}

\address{$^2$Department of Pure Mathematics, Center of Excellence in Analysis on Algebraic Structures (CEAAS), Ferdowsi University of Mashhad, P. O. Box 1159, Mashhad 91775, Iran.}
\email{moslehian@um.ac.ir; moslehian@yahoo.com}

\renewcommand{\subjclassname}{\textup{2020} Mathematics Subject Classification}


\begin{abstract}
Suppose that $\mathscr{E}$ and $\mathscr{F}$ are Hilbert $C^*$-modules.  We present a power-norm $\left(\left\|\cdot\right\|^{\mathscr{E}}_n:n\in\mathbb{N}\right)$ based on $\mathscr{E}$ and obtain some of its fundamental properties. We introduce a new definition of the absolutely $(2,2)$-summing operators from $\mathscr{E}$ to $\mathscr{F}$,  and denote the set of such operators by  $\tilde{\Pi}_2(\mathscr{E},\mathscr{F})$ with the convention $\tilde{\Pi}_2(\mathscr{E})=\tilde{\Pi}_2(\mathscr{E},\mathscr{E})$. It is known that the class of all Hilbert--Schmidt operators on a Hilbert space $\mathscr{H}$ is the same as the space $\tilde{\Pi}_2(\mathscr{H})$. We show that the class of Hilbert--Schmidt operators introduced by Frank and Larson coincides with the space $\tilde{\Pi}_2(\mathscr{E})$ for a  countably generated Hilbert $C^*$-module $\mathscr{E}$ over a unital commutative $C^*$-algebra.  These results motivate us to investigate the properties of the space $\tilde{\Pi}_2(\mathscr{E},\mathscr{F})$.
\end{abstract}
\maketitle
\section{Introduction and preliminaries}
 
In 2012, Dales and Polyakov \cite{5} introduced the notions of multi-norm and dual multi-norm based on a normed space. Although there are some analogies between this theory and the theory of operator spaces, there are serious differences between them. For example, the theory of multi-normed spaces deals with the general Banach spaces $L^p$ for $1\leq p\leq \infty$, while the other theory is initiated for $L^2$-spaces. After the seminal work \cite{5}, several mathematicians have explored and generalized these spaces. For example, the second author and Dales \cite{MOS} investigated the perturbation of mappings between multi-normed spaces; and Blasco \cite{bbb} worked on several aspects of power-normed spaces; see \cite{WIC} for some open problems in this content.

We recall some definitions and notation from \cite{lance, MT, mor}.

Throughout this paper, let $\mathcal{E}$ be a normed space. For $k\in\mathbb{N}$, we denote the linear space $\mathcal{E}\oplus \dots \oplus \mathcal{E}$ consisting of $k$-tuples $(x_1,\ldots,x_k)$ by $\mathcal{E}^k$, where $x_1,\ldots, x_k \in \mathcal{E}$. As usual, the linear structure on $\mathcal{E}^k$ is defined coordinatewise. Let ${\mathcal{E}}_{[1]}$  stand for the closed unit ball of $\mathcal{E}$. We assume that $\mathfrak{A}$ is a \emph{$C^*$-algebra} equipped with its usual partial order $\leq$ on the self-adjoint part of $\mathfrak{A}$. We set $\mathfrak{A}_+=\left\lbrace a\in\mathfrak{A}:a\geq0\right\rbrace$. . A subalgebra $\mathfrak{I}$ of a $C^*$-algebra $\mathfrak{A}$ is said to be \emph{(left) right ideal} if ($ba\in\mathfrak{I}$) $ab\in\mathfrak{I}$ for each $a\in\mathfrak{A}$ and $b\in\mathfrak{I}$. In the case where $\mathfrak{I}$ is a left and right ideal of $\mathfrak{A}$, we call it a \emph{two-sided ideal} of $\mathfrak{A}$.

We say that a  state $\tau$ is \emph{pure} if for any positive linear functional $\rho:\mathfrak{A}\rightarrow\mathbb{C}$ with $\rho\leq\tau$, there exists $0\leq t\leq1$ such that $\rho=t\tau$. The set of pure states on $\mathfrak{A}$ is denoted by  $\mathcal{PS}\left(\mathfrak{A}\right) $. In the case where $\mathfrak{A}$ is commutative, we observe that $\mathcal{PS}\left(\mathfrak{A}\right)=\Omega\left(\mathfrak{A}\right)$, where $\Omega\left( \mathfrak{A}\right)$ is  the character space of  $\mathfrak{A}$ equipped with the weak$\rm{^*}$ topology. 

By a (right) \emph{Hilbert $C^*$-module} over $\mathfrak{A}$, or shortly a Hilbert $\mathfrak{A}$-module, we mean a right $\mathfrak{A}$-module $\mathscr{E}$ equipped with an  $\mathfrak{A}$-valued inner product  $\langle \cdot, \cdot \rangle$
satisfying (i) $\langle x,x\rangle \geq 0$ with equality if and only if $x=0$; (ii) it is linear at the second variable; (iii) $\langle x,ya\rangle =\langle x,y\rangle a$; (iv) $\langle y,x\rangle=\langle x,y\rangle^{*}$ for all $x, y \in \mathscr{E}$, $a\in \mathfrak{A}$, and $\lambda\in \mathbb{C}$, and it is a Banach space endowed with the norm 
$\left\| x \right\|=\left\| \left\langle x, x \right\rangle \right\|^{1/2}$. Moreover set $\left| x \right|=  \left\langle x, x \right\rangle^{1/2}$. A $C^*$-algebra $\mathfrak{A}$ is a Hilbert $\mathfrak{A}$-module if we define $\left\langle a,b\right\rangle=a^*b$ for all $a,b\in\mathfrak{A}$.  The unit of $\mathfrak{A}$ is denoted by $1_{\mathfrak{A}}$. Inspired by \cite[Proposition 1.1]{lance}, there exists a useful version of the Cauchy--Schwarz inequality  that asserts $\left|\left\langle x, y \right\rangle \right|^2\leq\left\| x \right\|^2\left|y \right|^2$. Furthermore, $\left| xa \right|\leq\left\|x \right\|\left|a\right| $ for every $x,y \in \mathscr{E}$ and $a\in\mathfrak{A}$. However, it need not be the case that $\left|x+y\right|\leq\left|x \right|+\left|y\right|$ in general. From now on, we assume that $\mathscr{E}$, $\mathscr{F}$, and $\mathscr{G}$ are Hilbert $\mathfrak{A}$-modules. 

A map $T : \mathscr{E}\rightarrow \mathscr{F}$ is said to be \emph{adjointable} if there exists a map $T^* : \mathscr{F}\rightarrow \mathscr{E}$ such that $\left\langle Tx,y \right\rangle=\left\langle x,T^*y \right\rangle$ for all $x \in \mathscr{E}$ and $y \in \mathscr{F}$. The set of all adjointable maps from $\mathscr{E}$ into $\mathscr{F}$ is denoted by $\mathcal{L}(\mathscr{E},\mathscr{F})$, and we write $\mathcal{L}(\mathscr{E})$
for $\mathcal{L}(\mathscr{E},\mathscr{E})$. It is known from  \cite[p. 8]{lance} that $\mathcal{L}(\mathscr{E})$ is a $C^*$-algebra. Assume that $\tau$ is a positive linear functional on $\mathfrak{A}$. Set $\mathcal{N}_{\tau}=\left\lbrace x\in\mathscr{E}:\tau\left( \left\langle x,x\right\rangle\right) =0 \right\rbrace $. There is a well-defined inner product on  ${\mathscr{E}}/{\mathcal{N}_{\tau}}$ given by $\left\langle x+\mathcal{N}_{\tau},y+\mathcal{N}_{\tau}\right\rangle=\left\langle x,y\right\rangle$ for $x,y\in\mathscr{E}$. The Hilbert  completion of ${\mathscr{E}}/{\mathcal{N}_{\tau}}$ is denoted by $\mathscr{H}_{\tau}$. Let $T\in\mathcal{L}(\mathscr{E})$. Let us define the operator $\phi\left(T\right)$ on  $\mathscr{E}/\mathcal{N}_{\tau}$ by setting $\phi(T)\left(x+\mathcal{N}_{\tau}\right)=Tx+\mathcal{N}_{\tau}$ for $x\in\mathscr{E}$. It follows from \cite[Proposition 1.2]{lance} that $\left\|\phi(T)\left(x+\mathcal{N}_{\tau}\right)\right\|^2=\tau\left( \left\langle Tx,Tx\right\rangle\right)\leq\left\|T \right\|^2\tau\left( \left\langle x,x\right\rangle\right)=\left\|T\right\|^2\left\|x+\mathcal{N}_{\tau}\right\|^2$. Thus, the operator $\phi(T)$ has a unique extention  $\phi_{\tau}(T)\in\mathcal{L}(\mathscr{H}_{\tau})$. It is easily seen that the map $\phi_{\tau}:\mathcal{L}(\mathscr{E})\rightarrow\mathcal{L}(\mathscr{H}_{\tau})$ is a $*$-homomorphism. 

For given $x \in \mathscr{E}$ and $y \in \mathscr{F}$, define ${\theta}_{y,x}:\mathscr{E} \rightarrow \mathscr{F}$ by ${\theta}_{y,x}(z)\mapsto y\left\langle x,z\right\rangle \quad (z \in \mathscr{E})$. Evidently ${\theta}_{y,x} \in \mathcal{L}(\mathscr{E},\mathscr{F})$ with $\left( {\theta}_{y,x}\right)^*={\theta}_{x,y}$ and $\left\|{\theta}_{y,x}\right\|\leq\left\|x \right\| \left\|y \right\| $. We denote  the linear subspace of $\mathcal{L}(\mathscr{E},\mathscr{F})$ spanned by $\left\lbrace {\theta}_{y,x}: x \in \mathscr{E}, y \in \mathscr{F} \right\rbrace $ by $\mathcal{F}(\mathscr{E},\mathscr{F})$, and its norm-closure is denoted by $\mathcal{K}(\mathscr{E},\mathscr{F})$. We write $\mathcal{K}(\mathscr{E})$ for $\mathcal{K}(\mathscr{E},\mathscr{E})$. An operator $U\in\mathcal{L}(\mathscr{E},\mathscr{F})$ is said to be unitary if
$U^*U=1_{\mathcal{L}(\mathscr{E})}$ and $UU^*=1_{\mathcal{L}(\mathscr{F})}$. The set of all unitaries from $\mathscr{E}$ to $\mathscr{F}$ is denoted by $\mathcal{U}(\mathscr{E},\mathscr{F})$. We shall abbreviate $\mathcal{U}(\mathscr{E},\mathscr{E})$ by $\mathcal{U}(\mathscr{E})$.

 A closed submodule $\mathscr{F}$ of $\mathscr{E}$ is said to be \emph{orthogonally complemented} if $\mathscr{E}=\mathscr{F}\oplus\mathscr{F}^{\perp}$, where $\mathscr{F}^{\perp}=\left\lbrace y\in\mathscr{E}:\left\langle x,y\right\rangle=0, x\in\mathscr{F}\right\rbrace$. Let $\mathscr{F}$ and $\mathscr{G}$ be closed submodules of $\mathscr{E}$. If $\mathscr{G}\subseteq\mathscr{F}^{\perp}$, then we denote it by $\mathscr{F}\perp\mathscr{G}$. For an operator $T\in\mathcal{L}(\mathscr{E},\mathscr{F})$, we have  $\rm{ker}(T)\perp\overline{\rm{ran}(T^*)}$. An operator $W\in\mathcal{L}(\mathscr{E},\mathscr{F})$ is a partial isometry if $W^*W$ is a projection. 

Let $\mathscr{H}$ and $\mathscr{K}$ be Hilbert spaces. Assume that $T\in\mathcal{L}\left(\mathscr{H},\mathscr{K}\right)$ . For $1<p$, define the \emph{Schatten $p$-norm} of $T$ by $\left\| T\right\|_{\left( p\right) }=\left(\mathrm {Tr}~ \left| T\right| ^{\left(p\right) }\right)^{\frac{1}{p}}$.  The  space of \emph{p-th Schatten class operators} is denoted by $\mathcal{L}^p\left(\mathscr{H},\mathscr{K}\right)=\left\lbrace T\in\mathcal{L}\left(\mathscr{H},\mathscr{K}\right): \left\| T\right\|_{\left( p\right)}<\infty \right\rbrace$.  Recall that the second Schatten class is nothing than the \emph{Hilbert--Schmidt operators}. Moreover, the first Schatten class is the space of \emph{trace class operators} \cite[pp. 59--65]{mor}. We write $\mathcal{L}^p(\mathscr{H})$ for $\mathcal{L}^p(\mathscr{H},\mathscr{H})$. Now, we recall some definitions and notations from \cite{5, 61}.

Let $\left(\mathcal{E},\left\|\cdot\right\|\right)$ be a normed space over $\mathbb{C}$. A \emph{power-norm} based on $\mathcal{E}$ is a sequence $(\left\|\cdot\right\|_n : n\in \mathbb{N})$ such that $\left\|\cdot\right\|_n$ is a norm on	${\mathcal{E}}^n$ for each $n \in \mathbb{N}$,	$\left\|x_0\right\|_1=\left\|x_0\right\|$, where $x_0 \in \mathcal{E}$, 	and the following Axioms 	(A1)--(A3) hold for each $n \in \mathbb{N}$ and $x=(x_1,\ldots,x_n) \in {\mathcal{E}}^n$:
\begin{enumerate}
\item[(A1)]
for each $\sigma \in \mathfrak{S}_n$, where $\mathfrak{S}_n$ is the set of all permutation of 	$\left\lbrace 1,\ldots,n\right\rbrace $, we have
\begin{equation}\nonumber
\left\| (x_{\sigma(1)},\ldots,x_{\sigma(n)})\right\|_n = \left\| x \right\|_n ;
\end{equation}
\item[(A2)]
for each $\alpha_1,\ldots,\alpha_n \in \mathbb{C}$, 
\begin{equation}\nonumber
\left\|(\alpha_1x_1,\ldots,\alpha_nx_n) \right\|_n\leq (\max_{1\leq i\leq n } \left| \alpha_i \right| )\left\| x \right\|_n;	
\end{equation}
\item[(A3)]
f	or each $x_1,\ldots,x_{n-1} \in \mathcal{E}$,
\begin{equation}\nonumber
\left\|(x_1,\ldots,x_{n-1},0) \right\|_n=\left\|(x_1,\ldots,x_{n-1}) \right\|_{n-1}.
\end{equation}
\end{enumerate}
In this case, $\left( \left\|\cdot\right\|_n: n\in\mathbb{N} \right)$ is named as a \emph{power-norm} based on $\mathcal{E}$, and\\ $\left( \left( {\mathcal{E}}^n,\left\|\cdot\right\|_n\right):n\in\mathbb{N} \right) $ is a \emph{power-normed space}. A power-norm $\left( \left\|\cdot\right\|_n: n\in\mathbb{N} \right)$ based on a Hilbert $\mathfrak{A}$-module $\mathscr{E}$ is called a \emph{$C^*$-power-norm} based on $\mathscr{E}$ if, in addition to	(A1)--(A3), we have
\begin{enumerate}
\item[(B2)]
for each 	$a_1,\ldots,a_n \in \mathfrak{A}$, 
\begin{equation*}
\left\|(x_1a_1,\ldots,x_na_n) \right\|_n\leq (\max_{1\leq i\leq n}\left\|a_i\right\| )\left\| x \right\|_n.	
\end{equation*}
\end{enumerate}
A power-norm is a \emph{multi-norm} based on $\mathcal{E}$, and $\left( \left( {\mathcal{E}}^n,\left\|\cdot\right\|_n\right):n\in\mathbb{N} \right) $ is a \emph{multi-normed space} if, in addition to	(A1)--(A3), we have
\begin{enumerate}
	\item[(A4)]
	for each 
	$x_1,\ldots,x_{n-1} \in \mathcal{E}$,
	\begin{equation} \nonumber
	\left\|\left( x_1,\ldots,x_{n-1},x_{n-1}\right)\right\|_n\leq\left\|\left( x_1,\ldots,x_{n-2},x_{n-1}\right) \right\|_{n-1}.
	\end{equation}
\end{enumerate}
 For a power-normed space $\left(\left({\mathcal{E}}^n,\left\|\cdot\right\|_n\right):n\in\mathbb{N} \right)$, it holds that
\begin{equation}\nonumber
\max_{1\leq i\leq n}\left\| x_i \right\| \leq \left\|(x_1,\ldots,x_{n}) \right\|_{n}\leq \sum_{i=1}^{n}\left\| x_i \right\|.
\end{equation}
 Dales and Polyakov  \cite[Theorem 4.15]{5} presented the concept of \emph{Hilbert multi-norm} based on a Hilbert space $\mathscr{H}$ by 
\begin{equation}\label{111}\nonumber
\left\|\left(x_1,\ldots,x_n\right) \right\|^{\mathscr{H}}_n= \sup\left\|p_1x_1 +\dots+p_nx_n\right\|=
\sup\left(\left\|p_1x_1 \right\|^2+\dots+\left\|p_nx_n\right\|^2 \right)^{1/2}
\end{equation}
for each $x_1,\ldots,x_n\in\mathscr{E}$, where the supremum is taken over all families $(p_i, 1 \leq i \leq n)$ of mutually orthogonal projections summing to $1_{\mathcal{L}(\mathscr{H})}$.  In the second section, we extend the notion of Hilbert multi-norm to that of \emph{Hilbert $C^*$-multi-norm} based on $\mathscr{E}$.  In the third section, we start with the definition of \emph{$2$-summing norm} based on a normed space $\mathcal{E}$ such that
\begin{equation}\label{37}
\mu_{2,n}(x_1,\ldots,x_n)=\sup\left\lbrace \left( \sum_{i=1}^{n}\left| \lambda(x_i)\right|^2\right)^{1/2}:\lambda\in\mathcal{E}^{\prime}\right\rbrace,
\end{equation}
where $\mathcal{E}^{\prime}$ is the dual space of $\mathcal{E}$; see \cite[Definition 3.15]{5}. We introduce the following power-norms based on $\mathscr{E}$ such that
\begin{align*}
&\mu_{n}(x_1,\ldots,x_n) = \sup\left\lbrace \left\|\left(\sum_{i=1}^n \left|Tx_i\right|^{\,2} \right)^{1/2}\right\| : T\in\mathcal{L}\left(\mathscr{E},\mathfrak{A}\right)_{[1]}\right\rbrace,\\&\mu^*_{n}(x_1,\ldots,x_n) = \sup\left\lbrace \left\|\left(\sum_{i=1}^n \left|(Tx_i)^*\right|^{\,2} \right)^{1/2}\right\| : T\in\mathcal{L}\left(\mathscr{E},\mathfrak{A}\right)_{[1]}\right\rbrace.
\end{align*}
We transfer some properties of $2$-summing norm stated in \cite[p. 26]{mu} and \cite[p. 66]{17}, to the new versions. In the last section, we  introduce a new version of the absolutely  \emph{$(2,2)$-summing operators} from $\mathscr{E}$ to $\mathscr{F}$  stated in \cite[p. 68]{5}. We denote the set of these operators by $\tilde{\Pi}_2(\mathscr{E},\mathscr{F})$.  The space $\tilde{\Pi}_2(\mathscr{E},\mathscr{E})$ is abbreviated as $\tilde{\Pi}_2(\mathscr{E})$.  It is shown \cite[Proposition 10.1]{71} that $\tilde{\Pi}_2(\mathscr{H})=\mathcal{L}^2(\mathscr{H})$ and $\tilde{\pi}_2(u)=\left\|u\right\|_{(2)}$ for each $u\in\mathcal{L}^2(\mathscr{H})$. For a  countably generated Hilbert $C^*$-module $\mathscr{E}$ over a  unital commutative $C^*$-algebra $\mathfrak{A}$, Frank and Larson introduced the class of Hilbert--Schmidt operators in $\mathcal{L}(\mathscr{E})$ \cite[p. 21]{comm}, and  Kaad \cite[Remark 2.7]{kad}  omited the unital condition of $\mathfrak{A}$.

 In this paper, we show that $\tilde{\Pi}_2(\mathscr{E})$  coincides with the class of Hilbert--Schmidt operators in $\mathcal{L}(\mathscr{E})$. In 2021, Stern and van Suijlekom \cite{sche} introduced  Schatten classes for Hilbert $C^*$-modules over commutative $C^*$-algebras. The  classes form two-sided ideals of compact operators and are equipped with a Banach norm and a $C^*$-valued trace with interesting properties. These results encourage us to explore the properties of the space $\tilde{\Pi}_2(\mathscr{E},\mathscr{F})$. In what follows, for $T\in\mathcal{L}(\mathscr{E})$, we set $\tilde{\pi}_1(T):=\tilde{\pi}_2(\left|T\right|^{1/2})^2$,  and explore some  fundamental propertie	s of $\tilde{\Pi}_1(\mathscr{E}):=\left\lbrace T\in\mathcal{L}(\mathscr{E}):\tilde{\pi}_1(T)<\infty \right\rbrace$. Furthermore, we show that $\tilde{\pi}_p(T)\leq\sup_{\tau\in\mathcal{PS}\left( \mathfrak{A}\right)}\left\|\phi_{\tau}(T)\right\| _{\left(p\right)}$ for $T\in\mathcal{L}(\mathscr{E})$ and $p\in\left\lbrace 1,2\right\rbrace $. Finally, by making use of \cite[Theorem 15.3.7]{wegga} and \cite[Lemma 3.12]{pol}, we consider some conditions to obtain more properties of the spaces $\tilde{\Pi}_2(\mathscr{E},\mathscr{F})$ and $\tilde{\Pi}_1(\mathscr{E})$.

\section{Multi-norm based on a Hilbert $C^*$-module}
Let $\mathscr{E}$ be a Hilbert $\mathfrak{A}$-module. Set 
\begin{equation}\label{77}
\left\|\left(x_1,\ldots,x_n\right) \right\|^{\mathscr{E}}_n=\sup\left\|P_1x_1 +\dots+P_nx_n\right\|=\sup\left\| \left|P_1x_1 \right|^2+\dots+\left|P_nx_n \right|^2\right\|^{1/2},
\end{equation}
for each $x_1,\ldots,x_n\in\mathscr{E}$, where  the supremum is taken over all families $(P_i:1\leq i\leq n)$ of mutually orthogonal projections in $\mathcal{L}(\mathscr{E})$ summing to $1_{\mathcal{L}(\mathscr{E})}$ by allowing the possibility that $P_j=0$ for some $1 \leq j \leq n $. It  immediately follows that $\left(\left\|\cdot\right\|^{\mathscr{E}}_n:n\in\mathbb{N}\right)$ is a multi-norm based on $\mathscr{E}$. We call it the \emph{Hilbert $C^*$-multi-norm} based on $\mathscr{E}$.

\begin{example}\label{QQ}
Let $\mathfrak{A}$ be a unital $C^*$-algebra. Since $\mathcal{L}(\mathfrak{A})=\mathfrak{A}$ \cite[p. 10]{lance}, the power-norm  $\left(\left\|\cdot\right\|^{\mathscr{E}}_n:n\in\mathbb{N}\right)$	based on	$\mathscr{E}=\mathfrak{A}$	can be written as 
\begin{equation}\label{56}
\left\|(a_1,\ldots,a_n) \right\|^{\mathscr{E}}_n = \sup\left\|p_1a_1 +\dots+p_na_n\right\| 
\end{equation}
for $a_1,\ldots,a_n \in \mathfrak{A}$, where the supremum is taken over all families	$\left(p_i:1\leq i\leq n\right) $  of mutually orthogonal projections in	$\mathfrak{A}$	 summing to $1_{\mathfrak{A}}$ by allowing the possibility that 	$p_j=0$ for some	$1\leq j\leq n$. In the case when $\mathfrak{A}$ is unital and commutative,  it follows from \cite[Theorem 2.1.10]{mor} that $\mathfrak{A}=C(\Omega)$, where $\Omega$ is a compact Hausdorff space. Thus, at most one of the projections is $1_{\mathfrak{A}}$ and the others are zero. Therefore, $\left\|(a_1,\ldots,a_n) \right\|^{\mathscr{E}}_n=\max_{1\leq i\leq n}\left\|a_i\right\|$. 
\end{example}
We recall some definitions and notation from \cite[pp. 67--75]{32}.

Take $\mathcal{T} \in\mathbb{B}\left(\mathcal{E},\mathcal{F}\right) $, where $\mathbb{B}\left(\mathcal{E},\mathcal{F}\right)$ is the space of all bounded linear maps from $\mathcal{E}$ to $\mathcal{F}$. The $ n${th} amplification of  $\mathcal{T}$ from a power normed space $\left( \left({\mathcal{E}}^n,\left\|\cdot\right\|_n\right):n\in\mathbb{N}\right)  $ to a power normed space $\left( \left({\mathcal{F}}^n,\left\|\cdot\right\|_n\right):n\in\mathbb{N}\right)$ is defined by 
\begin{align*}
{\mathcal{T}}^{(n)} \left(x_1,\ldots,x_n \right) =\left(\mathcal{T}x_1,\ldots,\mathcal{T}x_n \right)\quad (n\in\mathbb{N}).
\end{align*}
 The $ n${th} amplification of $\mathcal{T}$ is bounded as a linear map from $ \left({\mathcal{E}}^n,\left\|\cdot\right\|_n\right)$ to $ \left({\mathcal{F}}^n,\left\|\cdot\right\|_n\right)$ such that
 \begin{equation}\label{ampl}
  \left\|\mathcal{T}\right\|\leq\left\|{\mathcal{T}}^{(n)}\right\|\leq n\left\|\mathcal{T}\right\|\quad(n\in\mathbb{N}). 
\end{equation}
Moreover, the sequence $\left(\left\|{\mathcal{T}}^{(n)}\right\|:n\in\mathbb{N} \right)$ is increasing.

\begin{definition}
	Let 	$\left( \left({\mathcal{E}}^n,\left\|\cdot\right\|_n\right):n\in\mathbb{N}\right)$	and	$\left( \left({\mathcal{F}}^n,\left\|\cdot\right\|_n\right):n\in\mathbb{N}\right)$ be power-normed spaces. Suppose	that	$\mathcal{T}\in \mathbb{B}\left(\mathcal{E},\mathcal{F}\right)$. 	Then	$\mathcal{T}$	is \emph{multi-bounded} with norm	$\left\|\mathcal{T}\right\|_{mb}$, if	
	\begin{equation}\label{am}
	\left\|\mathcal{T}\right\|_{mb} := \lim_{n\rightarrow \infty} \left\|{\mathcal{T}}^{(n)}\right\|< \infty.
	\end{equation}
\end{definition}

The collection of multi-bounded maps from $\mathcal{E}$ to $\mathcal{F}$ is denoted by $\mathcal{M}\left(\mathcal{E},\mathcal{F}\right) $.
If $\left( \left({\mathcal{E}}^n,\left\|\cdot\right\|_n\right):n\in\mathbb{N}\right)$ and $\left( \left({\mathcal{F}}^n,\left\|\cdot\right\|_n\right):n\in\mathbb{N}\right)$ are power-normed spaces, then \linebreak
$\left(\mathcal{M}\left(\mathcal{E},\mathcal{F}\right),\left\|\cdot\right\|_{mb}\right)$ is a normed space. Furthermore, if $\mathcal{F}$ is a Banach space, then $\left(\mathcal{M}\left(\mathcal{E},\mathcal{F}\right),\left\|\cdot\right\|_{mb}\right)$ is a Banach space \cite[Theorem 6.20]{5}.

Let   $\mathscr{E}$ and $\mathscr{F}$ be Hilbert $\mathfrak{A}$-modules. Suppose that 	$\left( \left({\mathscr{E}}^n,\left\|\cdot\right\|_n\right):n\in\mathbb{N}\right)$	and\\	$\left( \left({\mathscr{F}}^n,\left\|\cdot\right\|_n\right):n\in\mathbb{N}\right)$ are power-normed spaces. Then the collection of  multi-bounded adjointable operators from $\mathscr{E}$ to $\mathscr{F}$ is denoted by  $\mathcal{ML}(\mathscr{E},\mathscr{F})$. Since $\left\|T\right\|\leq\left\|T\right\|_{mb}$ for each $T\in\mathcal{L}(\mathscr{E},\mathscr{F})$ and $\mathcal{ML}(\mathscr{E},\mathscr{F})=\mathcal{L}(\mathscr{E},\mathscr{F})\cap\mathcal{M}\left(\mathscr{E},\mathscr{F}\right)$, it follows that $\mathcal{ML}(\mathscr{E},\mathscr{F})$ is a Banach space. The space $\mathcal{ML}(\mathscr{E},\mathscr{E})$ is abbreviated as $\mathcal{ML}(\mathscr{E})$.

We conclude this section with the following proposition, which was proved in the setting of Hilbert spaces in \cite[p. 127]{17}. We shall prove it for Hilbert $C^*$-modules. To this end, we employ the Russo--Dye theorem  asserting that if $\mathfrak{A}$ is a $C^*$-algebra, then $\mathfrak{A}_{[1]}$ is equal to $\overline{co}(\mathcal{U}(\mathfrak{A}))$, the norm-closed convex hull of $\mathcal{U}(\mathfrak{A})$. 

\begin{proposition}\label{xx}
	Let	$\left(\left\|\cdot\right\|^{\mathscr{E}}_n:n\in\mathbb{N}\right)$	be the Hilbert $C^*$-multi-norm based on $\mathscr{E}$. Then 
	\begin{equation}\label{90}
\mathcal{ML}(\mathscr{E})=\mathcal{L}(\mathscr{E}) \quad\text{with}\quad\left\|T\right\|_{mb} = \left\|T\right\| \quad\text{for each}\quad T\in \mathcal{L}(\mathscr{E}).
	\end{equation}
\end{proposition}
\begin{proof}
	Let $U\in\mathcal{U}\left(\mathscr{E}\right)$.  Fix $x_1,\ldots,x_n\in\mathscr{E}$ such that	$\left\|\left(x_1,\ldots,x_n \right) \right\|^{\mathscr{E}}_n\leq 1$. Then
	\begin{equation} \label{3}
	\left\|\left|P_1x_1 \right|^2+\dots+\left|P_nx_n \right|^2\right\| \leq 1
	\end{equation}
for some family $(P_i:1\leq i\leq n)$ of mutually orthogonal projections summing  to $1_{\mathcal{L}\left(\mathscr{E}\right)}$. Hence the family $(U^*P_iU:1\leq i\leq n)$ is consisting of mutually orthogonal projections summing to $1_{\mathcal{L}\left(\mathscr{E}\right)}$. It follows from \eqref{3} that
		\begin{align*} 
		\left\|\left(Ux_1,\ldots,Ux_n\right) \right\|^{\mathscr{E}}_n&=\sup\left\|\left|P_1Ux_1 \right|^2+\dots+\left|P_nUx_n\right|^2\right\|^{1/2}\\&=\sup\left\|\left|U^*P_1Ux_1 \right|^2+\dots+\left|U^*P_nUx_n\right|^2\right\|^{1/2}\leq 1.
	\end{align*}
     Therefore, $1=\left\|U\right\|\leq\left\|U^{(n)}\right\|\leq1$. Hence $\left\|U\right\|_{mb}=1$. Thus $U\in \mathcal{ML}(\mathscr{E})$. Since the unitaries span the $C^*$-algebra $\mathcal{L}(\mathscr{E})$, we arrive at	$\mathcal{ML}\left( \mathscr{E}\right) =\mathcal{L}(\mathscr{E})$.  From the Russo--Dye theorem, we observe that	$\mathcal{L}(\mathscr{E})_{[1]}=\overline{co}\left(\mathcal{U}\left(\mathscr{E}\right)\right)$.  Pick a nonzero operator $T$ in $\mathcal{L}(\mathscr{E})$. Put $S=T/\left\|T\right\|$. There exists a sequence $\left( S_n:n\in\mathbb{N}\right) $ in $\rm{co}\left(\mathcal{U}\left(\mathscr{E}\right)\right)$  such that $S_n\rightarrow S$ in the norm topology. We immediately see that  $\left\|S_n\right\|_{mb}\leq1$. Let $m\in\mathbb{N}$.  From \eqref{ampl}, we observe that  $\left\|S_n^{(m)}-S^{(m)}\right\|\leq m\left\|S_n-S\right\|$. Therefore, $\lim_{n\rightarrow \infty}\left\|S_n^{(m)}-S^{(m)}\right\|=0$. Thus
	\begin{equation}
	\left\|S^{(m)}\right\|=\lim_{n\rightarrow \infty}\left\|S_n^{(m)}\right\|\leq\sup_{n\in\mathbb{N}}\left\|S_n\right\|_{mb}\leq 1.
	\end{equation}
	Hence $1=\left\|S\right\|\leq\left\|S\right\|_{mb}=\lim_{m \to \infty}\left\|S^{(m)}\right\|\leq1$. Therefore $\left\|T\right\|_{mb}=\left\|T\right\|$.
\end{proof}

\section{New versions of $2$-summing norm }
As stated in \cite{17}, for a Banach space $\mathcal{E}$ and a nonzero power-normed space \linebreak$\left(\left({\mathcal{F}}^n,\left\|\cdot\right\|_n \right):n\in\mathbb{N}\right)$, there exists a norm on the space ${\mathcal{E}}^n$ defined by
\begin{equation}\label{l}
\left\| (x_1,\ldots, x_n)\right\|^{\mathcal{F}}_n= \sup\left\lbrace \left\|(Tx_1,\ldots,Tx_n)\right\|_n : T \in \mathbb{B}(\mathcal{E},\mathcal{F})_{[1]}\right\rbrace,
\end{equation}
where $x_1,\ldots,x_n \in \mathcal{E}$. Evidently, $(\left\|\cdot\right\|^{\mathcal{F}}_n: n \in \mathbb{N})$ is a power-norm based on $\mathcal{E}$ and  ${\mathcal{M}}(\mathcal{E},\mathcal{F}) = \mathbb{B}(\mathcal{E},\mathcal{F})$ with $\left\| T\right\|_{mb} = \left\| T\right\|$ for each $T \in {\mathcal{M}}(\mathcal{E},\mathcal{F})$, where ${\mathcal{M}}(\mathcal{E},\mathcal{F})$ is calculated with respect to the power-norm  $\left( \left\|\cdot\right\|^{\mathcal{F}}_n: n \in \mathbb{N}\right) $ based on $\mathcal{E}$. 
 
In this section, we introduce a power-norm on $\mathcal{L}(\mathscr{E},\mathfrak{A})$  by mimicking the power-norm introduced in  \cite{17}. To this end, we state some concepts from \cite{lance}; see also \cite{FRA1}.
 
Let $\mathscr{E}$ and $\mathscr{F}$ be Hilbert $\mathfrak{A}$-modules. The strict topology on $\mathcal{L}(\mathscr{E},\mathscr{F})$ is defined by the seminorms
\begin{equation*}
T\mapsto\left\|Tx\right\|\quad(x\in\mathscr{E}),\quad  T\mapsto\left\|T^*y\right\|\quad(y\in\mathscr{F}).
\end{equation*}
In virtue of \cite[p. 12]{lance}, the map
\begin{equation}\label{k}
\mathscr{E}\rightarrow\mathcal{K}(\mathscr{E}, \mathfrak{A})\quad x\mapsto T_x, 
\end{equation}
is isometric and surjective, where $T_x:\mathscr{E}\rightarrow\mathfrak{A}$ is defined by $y\mapsto\left\langle x,y\right\rangle$.
Now, we introduce a power-norm based on $\mathscr{E}$ depending on a power-norm based on $\mathfrak{A}$.
\begin{proposition}\label{18}
Let	$\mathscr{E}$ be a Hilbert $\mathfrak{A}$-module. Suppose that $\left(\left\|\cdot\right\|_n:n\in\mathbb{N}\right)$ is a power norm ($C^*$-power norm) based on $\mathfrak{A}$. Then
\begin{align*} 
\left\|(x_1,\ldots,x_n)\right\|^{\mathfrak{A}}_n   &=\sup\left\lbrace \left\|(\left\langle y,x_1\right\rangle,\ldots, \left\langle y,x_n\right\rangle)\right\|_n : y \in {\mathscr{E}}_{[1]}\right\rbrace\\&=\sup\left\lbrace \left\|(Tx_1,\ldots,Tx_n)\right\|_n : T \in{\mathcal{K}(\mathscr{E}, \mathfrak{A})}_{[1]}\right\rbrace\\&=\sup\left\lbrace \left\|(Tx_1,\ldots,Tx_n)\right\|_n : T \in{\mathcal{L}(\mathscr{E}, \mathfrak{A})}_{[1]}\right\rbrace,
\end{align*}
  and $\left( \left\|\cdot\right\|^{\mathfrak{A}}_n:n\in\mathbb{N}\right)$	is a power norm ($C^*$-power norm) based on $\mathscr{E}$. Moreover 
\begin{equation}\label{89}
\mathcal{ML}(\mathscr{E},\mathfrak{A})=\mathcal{L}(\mathscr{E}, \mathfrak{A})\quad\text{with}\quad\left\|S\right\|_{mb} = \left\|S\right\| \quad\text{for each}\quad S\in \mathcal{L}(\mathscr{E},\mathfrak{A}).
\end{equation}
\end{proposition}
\begin{proof}
From $\left\|x \right\|^{\mathfrak{A}}_1=\sup\left\lbrace \left\|\left\langle x,y \right\rangle\right\|: y \in \mathscr{E}_{[1]}\right\rbrace=\left\|x \right\|$, where $x\in\mathscr{E}$, we immediately derive that $\left(\left\|\cdot\right\|^{\mathfrak{A}}_n:n\in\mathbb{N}\right)$ is a power norm ($C^*$-power norm) based on $\mathscr{E}$. Take  $T\in \mathcal{L}(\mathscr{E},\mathfrak{A})_{[1]}$. Since $\mathcal{K}(\mathscr{E},\mathfrak{A})_{[1]}$ is strictly dense in $\mathcal{L}(\mathscr{E},\mathfrak{A})_{[1]}$  \cite[Proposition 1.3]{lance}, there exists a net  $\left( y_{\lambda} :\lambda\in \Lambda\right)$ in ${\mathscr{E}}_{[1]}$ such that
\begin{equation}\label{7}
Tx=\lim_{\lambda}T_{\lambda}x=\lim_{\lambda}\left\langle y_{\lambda},x \right\rangle\quad (x\in\mathscr{E}).
\end{equation} 
We deduce from \eqref{ampl}  that
\begin{align*}
\left\|\left( Tx_1, \ldots, Tx_n\right) \right\|_n=\lim_{\lambda}\left\|(\left\langle y_{\lambda},x_1 \right\rangle,\ldots,\left\langle y_{\lambda},x_n \right\rangle)\right\|_n \quad (n\in\mathbb{N}).
\end{align*}
Thus, we observe that
\begin{align*}
\sup\left\lbrace \left\|(\left\langle y,x_1\right\rangle,\ldots, \left\langle y,x_n\right\rangle)\right\|_n : y \in {\mathscr{E}}_{[1]}\right\rbrace=\sup\left\lbrace \left\|(Tx_1,\ldots,Tx_n)\right\|_n : T \in {\mathcal{L}(\mathscr{E}, \mathfrak{A})}_{[1]}\right\rbrace.
\end{align*}
Suppose that $S\in\mathcal{L}(\mathscr{E},\mathfrak{A})$. Consider the net $\left( z_{\lambda} : {\lambda}\in{\Lambda}\right)$  as in equation \eqref{7} for the operator ${S}/{\left\|S\right\|}$. Fix $n\in\mathbb{N}$.
Pick $x_1,\ldots,x_n\in\mathscr{E}$ with $\left\|\left(x_1,\ldots, x_n \right)\right\|^{\mathfrak{A}}_n\leq 1$. Then 
\begin{align*} 
\left\|(Sx_1,\ldots,Sx_n)\right\|_n &=\left\|S\right\| \left\|\left( \frac{S}{\left\|S\right\|}x_1, \ldots, \frac{S}{\left\|S\right\|}x_n\right) \right\|_n \\ &=\left\|S\right\|\lim_{\lambda}\left\|(\left\langle z_{\lambda},x_1 \right\rangle,\ldots,\left\langle z_{\lambda},x_n \right\rangle)\right\|_n\\ &\leq\left\|S\right\|\left\|\left(x_1,\ldots x_n \right)\right\|^{\mathfrak{A}}_n\leq\left\|S\right\|.
\end{align*} 
Therefore, $\left\|S\right\|\leq\left\|S^{(n)}\right\|\leq\left\|S\right\|$. Hence $ \left\|S\right\|_{mb}=\lim_{n\rightarrow\infty}\left\|S^{(n)}\right\|=\left\|S\right\|$.
\end{proof}

\begin{definition} 
Let ${\mathscr{E}}$ be a Hilbert $\mathfrak{A}$-module. We write $\mathit{l}_n^2(\mathscr{E})$ for $\mathscr{E}^n$ as a Hilbert $\mathfrak{A}$-module with the inner product  given by
\begin{align*}
\left\langle(x_1,\ldots,x_n),(y_1,\ldots,y_n)\right\rangle_{\mathit{l}_n^2(\mathscr{E})}=\sum_{i=1}^{n} \left\langle x_i,y_i\right\rangle.
\end{align*} 
We shall abbreviate $\mathit{l}_n^2(\mathbb{C})$ by $\mathit{l}_n^2$.
\end{definition}

In the following definition, we present two power-norms $\mu_n$ and $\mu^*_n$ based on a Hilbert $\mathfrak{A}$-module ${\mathscr{E}}$ inspired by  the power-norm $\mu_{2,n}$; see \eqref{37}.
\begin{definition}
Let ${\mathscr{E}}$ be a Hilbert $\mathfrak{A}$-module. Given $x_1,\dots,x_n \in \mathscr{E}$, define 
\begin{equation}\label{muu}
\mu_{n}(x_1,\dots,x_n) = \sup\left\lbrace \left\|\left(\sum_{i=1}^n \left|Tx_i\right|^{2} \right)^{1/2}\right\| : T\in\mathcal{L}\left(\mathscr{E},\mathfrak{A}\right)_{[1]}\right\rbrace\,
\end{equation}
and furthermore,
\begin{equation}\label{muuu}
\mu^*_{n}(x_1,\dots,x_n) = \sup\left\lbrace \left\|\left(\sum_{i=1}^n \left|(Tx_i)^*\right|^{2} \right)^{1/2}\right\| : T\in\mathcal{L}\left(\mathscr{E},\mathfrak{A}\right)_{[1]}\right\rbrace\,.
\end{equation}
\end{definition}
 In the case where ${\mathfrak{A}}$ is commutative, $\mu^*_n=\mu_n$. Moreover, in the setting of Hilbert spaces, $\mu^*_n=\mu_n=\mu_{2,n}$. 
\begin{proposition}
Let ${\mathscr{E}}$ and ${\mathscr{F}}$ be  Hilbert $\mathfrak{A}$-modules. 
\begin{enumerate}
\item 
The sequence $\left(\mu_n: n\in\mathbb{N}\right)$ is a  power-norm based on $\mathscr{E}$ and
\begin{equation}\label{b}
\mu_{n}(x_1,\ldots,x_n)=\sup_{y\in\mathscr{E}_{[1]}}\left\|\left( \sum_{i=1}^{n}\left|\left\langle y,x_i\right\rangle\right|^2\right)^{1/2}\right\|.
\end{equation}
Moreover,
\begin{equation}\label{a}
\mu_{n}(Tx_1,\dots,Tx_n)\leq\left\|T\right\|\mu_{n}(x_1,\dots,x_n)\quad\left( T\in\mathcal{L}(\mathscr{E},\mathscr{F})\right). 
\end{equation} 
\item
The sequence $\left(\mu^*_n: n\in\mathbb{N}\right)$ is a  $C^*$-power-norm based on $\mathscr{E}$ and
\begin{equation}\label{bb}
\mu^*_{n}(x_1,\ldots,x_n)=\sup_{y\in\mathscr{E}_{[1]}}\left\|\left( \sum_{i=1}^{n}\left|\left\langle x_i,y\right\rangle\right|^2\right)^{1/2}\right\|.
\end{equation}
Moreover,
\begin{equation}\label{c}
\mu^*_{n}(Tx_1,\dots,Tx_n)\leq\left\|T\right\|\mu_{n}(x_1,\dots,x_n)\quad\left( T\in\mathcal{L}(\mathscr{E},\mathscr{F})\right). 
\end{equation} 
\end{enumerate}
\end{proposition}
\begin{proof}
(1) Consider the power-normed space  $\left( \left( \mathfrak{A}^n,\left\|\cdot\right\|_{\mathit{l}_n^2(\mathfrak{A})}\right) : n\in\mathbb{N}\right)$.  From \eqref{muu} and Proposition \ref{18}, we infer  that
\begin{align*}
\mu_{n}\left( x_1,\ldots,x_n\right) &=\sup\left\lbrace \left\|\left( Tx_1,\ldots,Tx_n\right) \right\|_{\mathit{l}_n^2(\mathfrak{A})}: T\in\mathcal{L}\left(\mathscr{E},\mathfrak{A}\right)_{[1]}\right\rbrace\\&=\sup_{y\in\mathscr{E}_{[1]}} \left\|(\left\langle y,x_1\right\rangle,\ldots, \left\langle y,x_n\right\rangle)\right\|_{\mathit{l}_n^2(\mathfrak{A})}\\&=\sup_{y\in\mathscr{E}_{[1]}}\left\|\left( \sum_{i=1}^{n}\left|\left\langle y,x_i\right\rangle\right|^2\right)^{1/2}\right\|.
\end{align*}
 Hence $\left(\mu_n:n\in\mathbb{N}\right)$ is a power-norm based on $\mathscr{E}$. Thus, for $T\in\mathcal{L}(\mathscr{E},\mathscr{F})$,
\begin{align*}
\mu_{n}(Tx_1,\ldots,Tx_n)&=\sup_{y\in\mathscr{E}_{[1]}}\left\|\left( \sum_{i=1}^{n}\left|\left\langle T^*y,x_i\right\rangle\right|^2\right)^{1/2}\right\|\\&=\left\|T\right\| \sup_{y\in\mathscr{E}_{[1]}}\left\|\left(\sum_{i=1}^{n}\left|\left\langle 
\frac{T^*}{\left\|T\right\|}y,x_i\right\rangle\right|^2\right)^{1/2}\right\|\\&\leq\left\|T\right\| \sup_{z\in\mathscr{E}_{[1]}}\left\|\left(\sum_{i=1}^{n}\left|\left\langle z,x_i\right\rangle\right|^2\right)^{1/2}\right\|=\left\|T\right\|\mu_{n}(x_1,\ldots,x_n).
\end{align*}
(2) Define a norm  $\left\|(b_1,\ldots,b_n)\right\|^*_n=\left\|(b^*_1,\ldots,b^*_n)\right\|_{\mathit{l}_n^2(\mathfrak{A})}$ in $\mathfrak{A}^n$.  Clearly $\left(\left\|\cdot\right\|^*_n:n\in\mathbb{N}\right)$ is a $C^*$-power-norm based on $\mathfrak{A}$. It follows from Proposition \ref{18} that
\begin{align*}
\mu^*_{n}(x_1,\ldots,x_n)&=\sup\left\lbrace \left\|\left( Tx_1,\ldots,Tx_n\right) \right\|^*_n: T\in\mathcal{L}\left(\mathscr{E},\mathfrak{A}\right)_{[1]}\right\rbrace\\&=\sup_{y\in\mathscr{E}_{[1]}} \left\|(\left\langle x_1,y\right\rangle,\ldots, \left\langle x_n,y\right\rangle)\right\|_{\mathit{l}_n^2(\mathfrak{A})}\\&=\sup_{y\in\mathscr{E}_{[1]}}\left\|\left( \sum_{i=1}^{n}\left|\left\langle x_i,y\right\rangle\right|^2\right)^{1/2}\right\|.
\end{align*}
Hence, $\left(\mu^*_n:n\in\mathbb{N}\right)$ is a $C^*$-power-norm based on $\mathscr{E}$. Thus, for $T\in\mathcal{L}(\mathscr{E},\mathscr{F})$,
\begin{align*}
\mu^*_{n}(Tx_1,\ldots,Tx_n)&=\sup_{y\in\mathscr{E}_{[1]}}\left\|\left( \sum_{i=1}^{n}\left|\left\langle x_i,T^*y\right\rangle\right|^2\right)^{1/2}\right\|\\&\leq\left\|T\right\| \sup_{y\in\mathscr{E}_{[1]}}\left\|\left(\sum_{i=1}^{n}\left|\left\langle 
x_i,\frac{T^*}{\left\|T\right\|}y\right\rangle\right|^2\right)^{1/2}\right\|\leq\left\|T\right\|\mu^*_{n}(x_1,\ldots,x_n).
\end{align*}
\end{proof}
\begin{example}
Let $\mathfrak{A}$ be a $C^*$-algebra and $\mathfrak{I}$ be a closed right ideal of $\mathfrak{A}$. Clearly $\mathfrak{I}$ is a closed submodule of $\mathfrak{A}$. Then
\begin{equation}\label{44}
\mu_n(b_1,\ldots,b_n)=\left\|(b_1,\ldots,b_n)\right\|_{\mathit{l}_n^2(\mathfrak{I})}
\end{equation}
 To see this, suppose that $b_1,\ldots,b_n\in\mathfrak{I}$. It follows from \cite[Theorem 3.1.2]{mor} that there exists an increasing net $\left(u_{\lambda}\right)_{\lambda\in\Lambda} $ of positive elements in $\mathfrak{I}_{[1]}$  such that $\lim_{\lambda}u_{\lambda}a=a$. Hence,
	\begin{align*}
	\left\|(b_1,\ldots,b_n)\right\|_{\mathit{l}_n^2(\mathfrak{I})}=\lim_{\lambda}\left\|(u_{\lambda}b_1,\ldots,u_{\lambda}b_n)\right\|_{\mathit{l}_n^2(\mathfrak{I})}&\leq\sup_{\left\|c\right\|=1 }\left\|\sum_{i=1}^n b^*_ic^*cb_i\right\|\\&=\mu_n(b_1,\ldots,b_n)\\&\leq \left\|(b_1,\ldots,b_n)\right\|_{\mathit{l}_n^2(\mathfrak{I})}.	
\end{align*}
\end{example}

 In this section, we employ $\mu^*_n$  and obtain some of its fundamental properties.
\begin{proposition}
Let $\mathscr{E}$ be a Hilbert $C^*$-module over a $C^*$-algebra $\mathfrak{A}$. Then
\begin{equation}\label{yi}
\mu^*_{n}(x_1,\ldots,x_n)=\min\left\lbrace\lambda>0:\left( \sum_{i=1}^{n}\left| \left\langle x_i,x\right\rangle\right|^2\right)^{1/2}\leq\lambda\left|x\right|\,\ \text{for all $x\in\mathscr{E}$}\right\rbrace
\end{equation}
for $x_1,\ldots,x_n\in\mathscr{E}$.
\end{proposition}
\begin{proof}
 We first assume that $\mathfrak{A}$ is unital. It follows from  the Cauchy--Schwarz inequality that $\left( \sum_{i=1}^{n}\left| \left\langle x_i,x\right\rangle\right|^2\right)^{1/2}\leq \left(\sum_{i=1}^{n}  \left\| x_i\right\|^2\right)^{1/2}\left|x\right|$ for all $x\in\mathscr{E}$. Thus the set stated in \eqref{yi} is nonempty. Suppose   $\lambda>0$ such that  $\left\| \left( \sum_{i=1}^{n}\left| \left\langle x_i,x\right\rangle\right|^2\right)^{1/2}\right\| \leq\lambda\left\|x\right\|$ for all $x\in\mathscr{E}$. Set $	T:\mathscr{E}\rightarrow\mathit{l}_n^2(\mathfrak{A})$ defined by $Tx=\sum_{i=1}^{n}\delta_i\left\langle x_i,x\right\rangle$, where $\delta_i=\left(0,\ldots,0,1_{\mathfrak{A}},0,\ldots,0 \right) $ with the unit being the $i$th entry. The operator $T$ is adjointable with $T^*:\mathit{l}_n^2(\mathfrak{A})\rightarrow\mathscr{E}$ defined by $T^*(a_1,\ldots,a_n)=\sum_{i=1}^{n}x_ia_i$. Since
\begin{align*}
\left\| \left\langle Tx,Tx\right\rangle\right\|=\left\| \sum_{i=1}^{n}\left\langle x,x_i\right\rangle\left\langle x_i,x\right\rangle\right\|\leq\lambda^2\left\|x\right\|^2,
\end{align*}
we arrive at $\left\|T\right\|\leq\lambda$. Thus $\left|Tx\right|\leq\lambda\left|x\right|$  \cite[Proposition 1.2]{lance} or, equivalently, $\left( \sum_{i=1}^{n}\left| \left\langle x_i,x\right\rangle\right|^2\right)^{1/2}\leq\lambda\left|x\right|$. From \eqref{bb}, we infer that
\begin{align*}
\mu^*_{n}(x_1,\ldots,x_n)&=\sup_{y\in\mathscr{E}_{[1]}}\left\|\left( \sum_{i=1}^{n}\left|\left\langle x_i,y\right\rangle\right|^2\right)^{1/2}\right\|\\&=\inf\left\lbrace\lambda>0:\left\| \left( \sum_{i=1}^{n}\left|\left\langle x_i,x\right\rangle\right|^2\right)^{1/2}\right\| \leq\lambda\left\|x\right\|\,\ \text{for all $x\in\mathscr{E}$}\right\rbrace\\&=\inf\left\lbrace\lambda>0:\left( \sum_{i=1}^{n}\left| \left\langle x_i,x\right\rangle\right|^2\right)^{1/2}\leq\lambda\left|x\right|\,\ \text{for all $x\in\mathscr{E}$}\right\rbrace.
\end{align*}
Clearly, the infimum is indeed a minimum. 

In the case when $\mathfrak{A}$ has no unit, we use the minimal unitization $\mathfrak{A}\oplus \mathbb{C}$ of $\mathfrak{A}$, and consider $\mathscr{X}$ as a Hilbert $\mathfrak{A}\oplus \mathbb{C}$-module. The right action is defined by $x\left(a+\eta\right)=xa+\eta x$ for each $x\in\mathscr{E}$, $a\in\mathfrak{A}$, and $\eta\in\mathbb{C}$. Thus, equation \eqref{yi} is established.
\end{proof}
Inspired by \cite[p. 26]{mu}, we present the following lemma  used in the proof of Theorem \ref{theorem}.
\begin{lemma}\label{star}
Let $\mathscr{E}$ be a Hilbert $\mathfrak{A}$-module. Then 
\begin{equation}
\mu^*_{n}(x_1,\ldots,x_n)=\sup\left\lbrace\left\| \sum_{i=1}^{n}x_ia_i\right\| : \left\| \left( a_1,\ldots,a_n\right)\right\|_{\mathit{l}^2_n(\mathfrak{A})}\leq 1\right\rbrace
\end{equation}
for $x_1,\ldots,x_n\in\mathscr{E}$.
\end{lemma}
\begin{proof}
Put $\alpha=\mu^*_{n}(x_1,\ldots,x_n)$ and $\beta=\sup\left\lbrace\left\| \sum_{i=1}^{n}x_ia_i\right\| : \left\| \sum_{i=1}^{n}\left|a_i\right|^2\right\|\leq 1 \right\rbrace$. If $\alpha=0$, then $\alpha\leq\beta$. Let $\alpha\neq0$. It follows from \eqref{bb} that there exists  $y\in\mathscr{E}_{[1]}$ such that $\gamma=\left\|\left( \sum_{i=1}^{n}\left|\left\langle x_i,y\right\rangle\right|^2\right)^{1/2}\right\|\neq 0$. Then
\begin{equation}\label{sigma}
\left\|\left(\sum_{i=1}^{n}\left|\left\langle x_i,y\right\rangle\right|^2\right)^{1/2}\right\|=\left\| \sum_{i=1}^{n}\left\langle y,x_i\right\rangle\frac{\left\langle x_i,y   \right\rangle}{\gamma}\right\|.
\end{equation}
Put $a_i={\left\langle x_i,y \right\rangle}/{\gamma}$. Hence $\left\| \sum_{i=1}^{n}\left|a_i\right|^2\right\|= 1$. It follows from \eqref{sigma} that
\begin{align*}
\left\|\left( \sum_{i=1}^{n}\left|\left\langle x_i,y\right\rangle\right|^2\right)^{1/2}\right\|=\left\|\left\langle y, \sum_{i=1}^{n}x_ia_i \right\rangle\right\|\leq \left\| \sum_{i=1}^{n}x_ia_i\right\|\leq\beta.
\end{align*}
Therefore $\alpha\leq\beta$. 

Conversely, fix $a_1,\ldots,a_n\in\mathfrak{A}$ such that $\left\|\sum_{i=1}^{n}\left|a_i\right|^2\right\|\leq 1$. For each $x_1,\ldots,x_n\in\mathscr{E}$, set $y_0={ \sum_{i=1}^{n}x_ia_i}/{\left\| \sum_{i=1}^{n}x_ia_i\right\|}$. Then
\begin{align*}
 \left\| \sum_{i=1}^{n}x_ia_i\right\|=\left\|\sum_{i=1}^{n}\left\langle y_0,x_i\right\rangle a_i\right\|&=\left\|\left\langle\left(\left\langle x_1,y_0\right\rangle_{\mathscr{E}},\ldots,\left\langle x_n,y_0\right\rangle_{\mathscr{E}} \right) ,\left(a_1,\ldots,a_n \right) \right\rangle_{\mathit{l}^2_n(\mathfrak{A})} \right\| \\&\leq\mu^*_{n}(x_1,\ldots,x_n) \tag{by Cauchy--Schwarz inequality},
\end{align*}
whence $\beta\leq\alpha$. Hence $\alpha=\beta$.
\end{proof}
 In virtue of \cite[p. 22]{32}, we present the following theorem.
\begin{theorem}\label{theorem}
Let $\mathscr{E}$ be a Hilbert $\mathfrak{A}$-module. For $x=(x_1,\ldots,x_n)\in\mathscr{E}^n$, consider
\begin{align*}
T_x:\mathit{l}_n^2(\mathfrak{A})\rightarrow\mathscr{E},\quad \left(a_1,\ldots,a_n\right)\mapsto\sum_{i=1}^{n}x_ia_i.
\end{align*}
The following statements hold.
\begin{enumerate}
	\item
	The map  $T:\left(\mathscr{E}^n,\mu^*_n\right)\rightarrow\mathcal{L}\left( \mathit{l}_n^2(\mathfrak{A}),\mathscr{E}\right)$ defined by $	x\mapsto T_x$ is isometric.
	\item 
	$\rm{ran}(T)=\mathcal{K}\left(\mathit{l}_n^2(\mathfrak{A}),\mathscr{E}\right)$.
	\item
	 If $\mathfrak{A}$ is unital, then the map is surjective.
\end{enumerate}
\end{theorem}
\begin{proof}
(1) Let $x=\left(x_1,\ldots,x_n\right)\in\mathscr{E}^n$. It is easy to verify that $T^*_x:\mathscr{E}\rightarrow\mathit{l}_n^2(\mathfrak{A})$ is defined by $ y\mapsto\left(\left\langle x_1,y\right\rangle,\ldots,\left\langle x_n,y\right\rangle\right)$. Thus $T_x\in\mathcal{L}\left( \mathit{l}_n^2(\mathfrak{A}),\mathscr{E}\right)$. It follows from Lemma \ref{star} that $\left\|T_x\right\|=\mu^*_n(x) $. Hence the map is isometric. 

(2) Let $x\in\mathscr{E}$ and $a=\left(a_1,\ldots,a_n\right)\in\mathit{l}_n^2(\mathfrak{A})$. Then
\begin{align}
\theta_{x,a}\left(b_1,\ldots,b_n\right)=x\sum_{i=1}^{n}a_i^*b_i=\sum_{i=1}^{n}\left( xa_i^*\right)b_i
\end{align}
for each $\left(b_1,\ldots,b_n\right)\in\mathit{l}_n^2(\mathfrak{A})$. Thus,  $\theta_{x,a}=T_{x_0}$, where $x_0=\left(xa_1^*,\ldots,xa_n^*\right)$. Clearly, 	$\mathcal{F}\left( \mathit{l}_n^2(\mathfrak{A}),\mathscr{E}\right)\subseteq\rm{ran}(T)$. Since $T$ is isometric, we conclude that  $\rm{ran}(T)$ is closed. Hence $\mathcal{K}\left( \mathit{l}_n^2(\mathfrak{A}),\mathscr{E}\right)\subseteq\rm{ran}(T)$.

Conversely, let $x=(x_1,\ldots,x_n)\in\mathscr{E}^n$. It follows from \cite[p. 13]{lance} that the map $S_{x_i}: \mathfrak{A}\rightarrow\mathscr{E}$ defined by $S_{x_i}(a)=x_ia$ lies in $\mathcal{K}\left( \mathfrak{A},\mathscr{E}\right)$. We observe that $T_x=[S_{x_1},\ldots,S_{x_n}] $, a $1\times n$ matrix over $\mathcal{K}\left(\mathfrak{A},\mathscr{E}\right)$. It is deduced from \cite[p. 11]{lance} that $T_x\in\mathcal{K}\left( \mathit{l}_n^2(\mathfrak{A}),\mathscr{E}\right)$. Hence,
$\rm{ran}(T)=\mathcal{K}\left( \mathit{l}_n^2(\mathfrak{A}),\mathscr{E}\right)$.

(3) Let $T\in\mathcal{L}\left( \mathit{l}_n^2(\mathfrak{A}),\mathscr{E}\right)$. Pick $\left(a_1,\ldots,a_n\right)\in\mathit{l}_n^2(\mathfrak{A}) $. Then
$T(a_1,\ldots,a_n)=\sum_{i=1}^{n}T\delta_ia_i$, where $\delta_i=\left(0,\ldots,0,1{\mathfrak{A}},0,\ldots,0 \right) $ with the unit being the $i$th entry. Thus $T=T_y$, where $y=\left(T\delta_1,\ldots, T\delta_n\right)$. 
\end{proof}
\begin{remark}
In the case when $\mathfrak{A}$ is nonunital, the map need not be surjective. For example, let $n=1$ and $\mathscr{E}=\mathfrak{A}=\mathbb{K}\left(\mathscr{H}\right)$, where the Hilbert space $\mathscr{H}$ is infinite dimentional. Then $\rm{ran}(T)=\mathcal{K}\left( \mathbb{K}\left(\mathscr{H}\right)\right)=\mathbb{K}\left(\mathscr{H}\right)\neq\mathbb{B}\left(\mathscr{H}\right)=\mathcal{L}\left( \mathbb{K}\left(\mathscr{H}\right)\right)$; see \cite[p. 82]{mor}.
\end{remark}
\section{New version of absolutely $(2,2)$-summing operators}
Let $\mathcal{E}$ and $\mathcal{F}$ be Banach spaces. For $T\in\mathbb{B}(\mathcal{E},\mathcal{F})$, set
\begin{align*}
\pi_{2}(T) : = \sup\left\{\left(\sum_{i=1}^n \left\|Tx_i\right\|^2\right)^{1/2}\ :x_1,\ldots, x_n\in \mathcal{E},\, \mu_{2,n}(x_1,\ldots, x_n)\leq 1, n\in\mathbb{N}\right\},
\end{align*}
where $\mu_{2,n}$ was defined at \eqref{37}. In the case where $\pi_2(T)<\infty$, the operator $T$ is said to be  absolutely $(2,2)$-summing operator; see \cite[p. 68]{5} and \cite[pp. 45--70]{71}. The set of these operators is denoted by $\Pi_2(\mathcal{E},\mathcal{F})$.  The space $\Pi_2(\mathcal{E},\mathcal{E})$ is abbreviated by $\Pi_2(\mathcal{E})$. In fact, $\Pi_2(\mathcal{E},\mathcal{F})=\mathcal{M}(\mathcal{E},\mathcal{F})$ and $\pi_2(T)=\left\|T\right\|_{mb}$ for $T\in\Pi_2(\mathcal{E},\mathcal{F})$. Here, the space $\mathcal{M}(\mathcal{E},\mathcal{F})$ is calculated with respect to the power-normed spaces $\left(\left( \mathcal{E}^n,\mu_{2,n}\right):n\in\mathbb{N}\right)  $ and $\left( \left( \mathcal{F}^n,\left\|\cdot\right\|_{\mathit{l}^2_n}\right) :n\in\mathbb{N}\right)$; see \cite[p. 59]{61}. Hence
$\left( \Pi_2(\mathcal{E},\mathcal{F}),\pi_2(\cdot)\right) $ is a Banach space whenever so is $\mathcal{F}$. 

  For a Hilbert space $\mathscr{H}$ Jaegermann \cite[Proposition 10.1]{71}  stated that 
\begin{equation}\label{HS}
\Pi_2(\mathscr{H})=\mathcal{L}^2(\mathscr{H}) \quad \text{and} \quad \pi_2(u)=\left\|u\right\|_{(2)}  \,\ \left( u\in\mathcal{L}^2(\mathscr{H})\right).
\end{equation}  
 This result motivates us to introduce the following class of adjointable operators from $\mathscr{E}$ to $\mathscr{F}$.
\begin{definition}\label{pp}
Let $\mathscr{E}$ and $\mathscr{F}$ be Hilbert $\mathfrak{A}$-modules. For $T \in\mathcal{L}(\mathscr{E},\mathscr{F})$, we define
\begin{equation}\label{1}
\tilde{\pi}_2(T) : = \sup\left\lbrace \left\|\sum_{i=1}^n \left|Tx_i\right| ^2\right\|^{1/2} :x_1,\ldots,x_n\in\mathscr{E},\, \mu_n(x_1,\ldots,x_n)\leq 1, n\in\mathbb{N}\right\rbrace.
\end{equation}	
\end{definition}
 The set of adjointable operators with $\tilde{\pi}_2(T) < \infty$ is denoted by $\tilde{\Pi}_{2}(\mathscr{E},\mathscr{F})$. In fact, $\tilde{\Pi}_{2}(\mathscr{E},\mathscr{F})=\mathcal{ML}(\mathscr{E},\mathscr{F})$  and $\tilde{\pi}_{2}(T)=\left\|T\right\|_{mb}$ for $T\in\tilde{\Pi}_{2}(\mathscr{E},\mathscr{F})$. Here, the space $\mathcal{ML}(\mathscr{E},\mathscr{F})$ is calculated with respect to the power-normed spaces $\left(\left( \mathscr{E}^n,\mu_n\right):n\in\mathbb{N}\right)  $ and $\left( \left( \mathscr{F}^n,\left\|\cdot\right\|_{\mathit{l}^2_n(\mathscr{F})}\right) :n\in\mathbb{N}\right)$. Hence $\left(\tilde{\Pi}_{2}(\mathscr{E},\mathscr{F}),\tilde{\pi}_{2}(\cdot)\right) $ is a Banach space. 
The proof of the next result is straightforward. However, we prove it for the sake of completeness.
\begin{proposition}\label{p2}
Let $\mathscr{E},\mathscr{F}$, and $\mathscr{G}$ be Hilbert $\mathfrak{A}$-modules.
\begin{enumerate}
\item 
For  $x\in\mathscr{E}$ and $y\in\mathscr{F}$, it holds that $\tilde{\pi}_2(\theta_{y,x})\leq\left\|x\right\|\left\|y\right\|$ and that $\mathcal{F}(\mathscr{E},\mathscr{F})\subseteq\tilde{\Pi}_{2}(\mathscr{E},\mathscr{F})$.
\item 
For  $T\in \mathcal{L}(\mathscr{E},\mathscr{F})$ and $S\in \mathcal{L}(\mathscr{F},\mathscr{G})$,  it holds that $\tilde{\pi}_2(ST)\leq \left\|S\right\|\tilde{\pi}_2(T)$ and that $\tilde{\pi}_2(ST)\leq \tilde{\pi}_2(S)\left\|T\right\|$.
\end{enumerate}
\end{proposition}
\begin{proof}
(1) Let $x_1,\ldots,x_n\in\mathscr{E}$ with $\mu_n(x_1,\ldots,x_n)\leq1$. It follows from \eqref{b} that
\begin{align*}
\left\|\sum_{i=1}^{n}\left\langle\theta_{y,x}x_i,\theta_{y,x}x_i\right\rangle \right\|=\left\|\sum_{i=1}^{n}\left\langle x_i,x\right\rangle\left\langle y,y\right\rangle\left\langle x,x_i\right\rangle \right\|\leq \left\|y\right\|^2 \left\|x\right\|^2. 
\end{align*} 
Thus $\tilde{\pi}_2(\theta_{y,x})\leq\left\|x\right\|\left\|y\right\|$.

(2) For $x_1,\ldots, x_n\in \mathscr{E}$ with $\mu_n(x_1,\ldots, x_n)\leq 1$, we have
\begin{equation}
\left\|\left(\sum_{i=1}^n\left|STx_i\right|^2\right)\right\|^{1/2}\leq\left\|S\right\|\left\|\left(\sum_{i=1}^n\left|Tx_i\right|^2\right)\right\|^{1/2} 
\end{equation}
since  $\left|STx_i\right|^2\leq\left\|S \right\|^2\left|Tx_i\right|^2$  \cite[Proposition 1.2]{lance}. Thus	$\tilde{\pi}_2(ST)\leq \left\|S\right\|\tilde{\pi}_2(T)$. Moreover,
\begin{equation}
\left\|\left(\sum_{i=1}^n\left|STx_i\right|^2\right)\right\|^{1/2}\leq	 \left\|\left(\sum_{i=1}^n\left|S\left( \frac{T}{\left\|T\right\| }x_i\right) \right|^2\right)\right\|^{1/2}\left\|T\right\|.
\end{equation}
We deduce from \eqref{b}  that $\mu_n(\frac{T}{\left\|T\right\| }x_1,\ldots, \frac{T}{\left\|T\right\| }x_n)\leq 1$. Thus $\tilde{\pi}_2(ST)\leq \tilde{\pi}_2(S)\left\|T\right\|$.
\end{proof}
From Proposition \ref{p2}, we obtain that  $\tilde{\Pi}_2(\mathscr{E})$ is a two-sided ideal of $\mathcal{L}(\mathscr{E})$. 
\begin{example}\label{ex}
Let $\mathfrak{A}$ be a $C^*$-algebra and $\mathfrak{I}$ be a closed right ideal of $\mathfrak{A}$. Then $\tilde{\pi}_{2}\left(T\right)=\left\|T\right\|$ for each $T\in\mathcal{L}\left(\mathfrak{I}\right)$. To see this, suppose that $b_1,\ldots,b_n\in\mathfrak{I}$ with $\mu_n(b_1,\ldots,b_n)\leq1$. Then
\begin{align*}
\left\|\sum_{i=1}^n \left\langle 1_{\mathcal{L}\left(\mathfrak{I}\right) }b_i,1_{\mathcal{L}\left(\mathfrak{I}\right) }b_i\right\rangle\right\|^{1/2}=\left\|(b_1,\ldots,b_n)\right\|_{\mathit{l}_n^2(\mathfrak{I})}=\mu_n(b_1,\ldots,b_n)\leq1\tag{by \eqref{44}}.
\end{align*}
Hence, $\tilde{\pi}_{2}\left(1_{\mathcal{L}\left(\mathfrak{I}\right)}\right)\leq1$. It follows from Proposition\ref{p2}(2)  that 
\begin{align*}
 \left\|T\right\|\leq\tilde{\pi}_{2}\left(T\right)\leq \left\|T\right\|\tilde{\pi}_{2}\left(1_{\mathcal{L}\left(\mathfrak{I}\right)}\right)\leq\left\| T\right\|.	
\end{align*}
Therefore,  $\tilde{\pi}_{2}\left(T\right)=\left\| T\right\|$ and $\tilde{\Pi}_{2}\left(\mathfrak{I}\right)=\mathcal{L}\left(\mathfrak{I}\right)$.	
\end{example}
 A countable subset $\mathbb{S}$ of $\mathscr{E}$ is a set of generators of $\mathscr{E}$ (as a Banach $\mathfrak{A}$-module), if the $\mathfrak{A}$-linear span of $\mathbb{S}$ is norm-dense in $\mathscr{E}$. 
 The following definition is borrowed from \cite[Definition 2.3]{sche}.
\begin{definition}
Let  $\mathscr{E}$ be a  countably generated Hilbert $C^*$-module over  a  $C^*$-algebra $\mathfrak{A}$.  A sequence $\left( f_i :i\in\mathbb{N}\right) $ of elements  of $\mathscr{E}$ is said to be a \emph{frame} if
\begin{equation}\label{e}
\left\langle x,y\right\rangle=\sum_{i=1}^{\infty}\left\langle x,f_i\right\rangle\left\langle f_i,y\right\rangle,
\end{equation}
in the norm topology, for all $x,y\in\mathscr{E}$. Such objects were called \emph{standard normalized frames} in \cite[p. 21]{comm}. 
\end{definition}

 Frank and Larson \cite[p. 21]{comm} introduced Hilbert--Schmidt operators on a countably generated Hilbert $\mathfrak{A}$-module $\mathscr{E}$ over a unital commutative $C^*$-algebra $\mathfrak{A}$.
 Kasparov’s theorem \cite[Theorem 1]{kas} states that such a Hilbert $C^*$-module $\mathscr{E}$ has a frame $\left(  f_i : i\in\mathbb{N}\right)  $. The operator $T\in\mathcal{L}(\mathscr{E})$ is (weakly) Hilbert--Schmidt if the sum $\sum_{i=1}^{\infty}\left\langle Tf_i,Tf_i\right\rangle$ \emph{weakly} converges in $\mathfrak{A}^{\prime\prime}$, where $\mathfrak{A}^{\prime\prime}$ is the bidual of $\mathfrak{A}$ as a von Neumann algebra; see \cite[Theorem 2.4]{Ta}. This definition is justified by the fact that if the sum $\sum_{i=1}^{\infty}\left\langle Tf_i,Tf_i\right\rangle$  converges weakly, then the sum $\sum_{i=1}^{\infty}\left\langle Tg_i,Tg_i\right\rangle$  also converges weakly and gives the same value in $\mathfrak{A}^{\prime\prime}$, where  $\left(  f_i : i\in\mathbb{N}\right)$ and $\left(  g_i : i\in\mathbb{N}\right)$ are two frames. Kaad \cite[Remark 2.7]{kad}  omited the unital condition in the above statements. 

The following lemma immediately follows from \cite[Theorems 4.1.1 and 4.2.4]{mor} and \cite[Theorem 2.4]{Ta}.

\begin{lemma}\label{79}
	Let $(a_\lambda:\lambda\in\Lambda)$ be a net of self-adjoint elements in $\mathfrak{A}$. Then $(a_\lambda:\lambda\in\Lambda)$ is weakly converges in $\mathfrak{A}^{\prime\prime}$ if it is increasing and bounded above.
\end{lemma}
\begin{theorem}\label{ffrr}
Let $\mathscr{E}$ be a countably generated Hilbert $\mathfrak{A}$-module over a commutative $C^*$-algebra $\mathfrak{A}$ with frame $\left(  f_i : i\in\mathbb{N}\right)$. Then $\sum_{i=1}^{\infty}\left\langle Tf_i,Tf_i\right\rangle$  converges weakly if and only if $T\in \tilde{\Pi}_2(\mathscr{E})$. Moreover, $\tilde{\pi}_2(T)=\left\| \sum_{i=1}^{\infty}\left\langle Tf_i,Tf_i\right\rangle\right\|^{1/2} $ for each  $T\in \tilde{\Pi}_2(\mathscr{E})$.
\end{theorem}
\begin{proof}
	Let $T\in\tilde{\Pi}_2(\mathscr{E})$. From \eqref{e}, we get
	\begin{align*}
	\left\|\sum_{i=1}^{n}\left\langle x,f_i\right\rangle\left\langle f_i,x\right\rangle \right\|^{1/2}\leq\left\|x \right\|\quad(x\in\mathscr{E}).  
	\end{align*}
	It follows that $\mu_n(f_1,\ldots,f_n)\leq1$. Thus $\left\| \sum_{i=1}^{n}\left\langle Tf_i,Tf_i\right\rangle\right\|^{1/2}\leq \tilde{\pi}_2(T)$. Put $a_n=\sum_{i=1}^{n}\left\langle Tf_i,Tf_i\right\rangle$. The sequence $(a_n:n\in\mathbb{N})$ is increasing and bounded above in $\mathfrak{A}$. Hence $\sum_{i=1}^{\infty}\left\langle Tf_i,Tf_i\right\rangle$ converges weakly in $\mathfrak{A}^{\prime\prime}$ in virtue of Lemma \ref{79}.  
	
	Conversely, let $\sum_{i=1}^{\infty}\left\langle Tf_i,Tf_i\right\rangle$ converge weakly in $\mathfrak{A}^{\prime\prime}$. Pick $n\in\mathbb{N}$ and $x_1,\ldots,x_n\in\mathscr{E}$ with $\mu_n(x_1,\ldots,x_n)\leq1$. Then
	\begin{align*}
	\sum_{j=1}^{n}\left\langle Tx_j,Tx_j\right\rangle&=\sum_{j=1}^{n}\sum_{i=1}^{\infty}\left\langle Tx_j,f_i\right\rangle\left\langle f_i,Tx_j\right\rangle\quad\tag{\text{by} \eqref{e}}\\&=\sum_{i=1}^{\infty}\sum_{j=1}^{n}\left\langle T^*f_i,x_j\right\rangle\left\langle x_j,T^*f_i\right\rangle\quad\tag{\text{by the commutativity of} $\mathfrak{A}$}\\&\leq\sum_{i=1}^{\infty}\left\langle T^*f_i,T^*f_i\right\rangle\tag{\text{by} \eqref{yi} and the fact that $\mu_n=\mu_n^*$}\\&=\sum_{i=1}^{\infty}\left\langle Tf_i,Tf_i\right\rangle\tag{\text{by} \cite[Proposition 4.8]{comm}}.
	\end{align*}
	Thus $\tilde{\pi}_2(T)\leq \left\| \sum_{i=1}^{\infty}\left\langle Tf_i,Tf_i\right\rangle\right\|^{1/2}$. Therefore $T\in\tilde{\Pi}_2(\mathscr{E})$. Hence\\ $\tilde{\pi}_2(T)=\left\| \sum_{i=1}^{\infty}\left\langle Tf_i,Tf_i\right\rangle\right\|^{1/2}$. 
\end{proof}
From the above theorem and the fact that  $\sum_{i=1}^{\infty}\left\langle Tf_i,Tf_i\right\rangle=\sum_{i=1}^{\infty}\left\langle T^*f_i,T^*f_i\right\rangle$ \cite[Proposition 4.8]{comm}, we conclude that $\left(\tilde{\Pi}_2(\mathscr{E}),\tilde{\pi}_2(\cdot)\right)$ is a Banach $*$-algebra.

For $T\in\mathcal{L}(\mathscr{E})$, set $\tilde{\pi}_1(T):=\tilde{\pi}_2(\left|T\right|^{1/2})^2$. Define
\begin{eqnarray}
\tilde{\Pi}_1(\mathscr{E})=\left\lbrace T\in\mathcal{L}(\mathscr{E}):\tilde{\pi}_1(T)<\infty \right\rbrace.
\end{eqnarray}
In fact, from Definition \ref{pp}, we get
\begin{equation}\label{absol}
\tilde{\pi}_1(T)=\sup\left\lbrace\left\|\sum_{i=1}^n \left\langle \left| T\right| x_i,x_i\right\rangle \right\| :x_1,\ldots, x_n\in \mathscr{E},\, \mu_n(x_1,\ldots, x_n)\leq 1, n\in\mathbb{N}\right\rbrace.
\end{equation}
 Suppose that  $\mathfrak{I}$ is a closed right ideal of $\mathfrak{A}$. In virtue of Example \ref{ex}, $\tilde{\pi}_{1}\left(T\right)=\left\|T\right\|$ for each $T\in\mathcal{L}\left(\mathfrak{I}\right)$.  Hence $\tilde{\Pi}_{1}\left(\mathfrak{I}\right)=\mathcal{L}\left(\mathfrak{I}\right)$. In the case where $\mathscr{H}$ is a Hilbert space, we have $\tilde{\Pi}_{1}\left(\mathscr{H}\right)=\mathcal{L}^1(\mathscr{H})$ and $\tilde{\pi}_{1}\left(T\right) =\left\|T\right\|_{(1)}$ for each $T\in\mathcal{L}^1(\mathscr{H})$.
 \begin{example}
Suppose that $\mathscr{E}$ is a countably generated Hilbert $\mathfrak{A}$-module over a commutative $C^*$-algebra $\mathfrak{A}$  with frame $\left(  f_i : i\in\mathbb{N}\right)$. It follows from Theorem \ref{ffrr} that 
 	 \begin{align*}
 	 \tilde{\pi}_{1}\left(T\right)=\tilde{\pi}_2(\left|T\right|^{1/2})=\left\|\sum_{i=1}^{\infty}\left\langle \left| T\right| f_i,f_i\right\rangle\right\|,
 	 \end{align*}
 	 where $T\in\tilde{\Pi}_{1}\left(\mathscr{E}\right)$. It follows from \cite[Definition 3.2 and Theorem 3.28]{sche} that  $\tilde{\pi}_{1}\left(T\right)=\sup_{\tau\in\Omega\left( \mathfrak{A}\right)}\left\|\phi_{\tau}(T)\right\| _{\left(1\right)}$ and $\tilde{\pi}_{1}\left(T\right)=\tilde{\pi}_{1}\left(T^*\right)$ . Moreover,  $\tilde{\Pi}_{1}\left(\mathscr{E}\right)$ is a Banach space.
 \end{example}
   We intend to show that $\left(\tilde{\Pi}_1(\mathscr{E}),\tilde{\pi}_1(\cdot)\right) 
 $ is a Banach space in general. To this end, we need the following lemma.
\begin{lemma}\cite[Theorem 4.2]{triangle}\label{hh}
	Let $\mathfrak{A}$ be a unital $C^*$-algebra. For each $a$ and $b$ in $\mathfrak{A}$ and $\varepsilon>0$, there are unitaries $u$ and $v$ in $\mathfrak{A}$ such that
	\begin{equation}
	\left|a+b\right|\leq u^*\left|a\right| u+v^*\left|b\right| v+\varepsilon 1_{\mathfrak{A}}.  
	\end{equation}
\end{lemma}

\begin{proposition}\label{ph}
	Let $\lambda\in\mathbb{C}$ and let $T,S\in \mathcal{L}(\mathscr{E})$. Then
	\begin{enumerate}
		\item
		$\tilde{\pi}_1(T+S)\leq\tilde{\pi}_1(T)+\tilde{\pi}_1(S)$ and $\tilde{\pi}_1(\lambda T)=\left|\lambda \right|\tilde{\pi}_1(T)$.
		\item
		$\left\|T\right\|\leq\tilde{\pi}_1(T)$.
		\item
		$\tilde{\pi}_1(TS)\leq \left\|T\right\|\tilde{\pi}_1(S)$.
		\item
		$\left(\tilde{\pi}_2(T)\right) ^2\leq\left\|T \right\|\tilde{\pi}_1(T)$.
	\end{enumerate}
\end{proposition}
\begin{proof}
	(1)  Given $\varepsilon>0$, pick $x_1,\ldots,x_n\in\mathscr{E}$ such that $\mu_n(x_1,\ldots, x_n)\leq 1$. Let  $T,S\in\mathcal{L}(\mathscr{E})$. It follows from Lemma \ref{hh} that there are unitaries $U$ and $V$ in $\mathcal{L}(\mathscr{E})$ such that
	\begin{equation}\label{q}
	\left|T+S\right|\leq U^*\left|T\right| U+V^*\left|S\right| V+\frac{\varepsilon}{n}l_{\mathcal{L}\left(\mathscr{E}\right)}.  
	\end{equation}  
	Using  \eqref{q}, we have
	\begin{align*}
	\left\|\sum_{i=1}^n\left\langle\left|T+S\right|x_i,x_i\right\rangle\right\|\leq\left\|\sum_{i=1}^n\left\langle\left|T\right|Ux_i,Ux_i\right\rangle\right\| +\left\|\sum_{i=1}^n\left\langle\left|S\right|Vx_i,Vx_i\right\rangle\right\|+\varepsilon.
	\end{align*}
	Since $\mu_n(Ux_1,\ldots, Ux_n)\leq 1$ and $\mu_n(Vx_1,\ldots, Vx_n)\leq 1$, we have 	$\tilde{\pi}_1(T+S)\leq\tilde{\pi}_1(T)+\tilde{\pi}_1(S)$. The equality $\tilde{\pi}_1(\lambda T)=\left|\lambda \right|\tilde{\pi}_1(T)$ can be easily proved.
	
	Clause	(2) follows from the corresponding property for the norm $\tilde{\pi}_2(\cdot)$, and clause (3) follows from   Proposition \ref{p2}. 
	
	(4) Suppose that $S\in\mathcal{L}(\mathscr{E})_{[1]}$. Let	$x_1,\ldots,x_n\in\mathscr{E}$  such that $\mu_n(x_1,\ldots, x_n)\leq 1$. Then
	\begin{align*}
	\left\|\sum_{i=1}^n\left\langle Sx_i,Sx_i\right\rangle\right\|\leq\left\|\sum_{i=1}^n\left\langle\left|S\right|\left|S\right|^{1/2}x_i,\left|S\right|^{1/2}x_i\right\rangle\right\|\leq\tilde{\pi}_1(S).
	\end{align*}
	Thus $\left( \tilde{\pi}_2(S)\right)^2\leq\tilde{\pi}_1(S)$. Hence clause (4) follows. 
\end{proof}
Clause (3) shows that $\tilde{\Pi}_1(\mathscr{E})$ is a left ideal of $\mathcal{L}(\mathscr{E})$. Clause (4) shows that $\tilde{\pi}_2(T)\leq\tilde{\pi}_1(T)$.  As a result, $\tilde{\Pi}_1(\mathscr{E})\subseteq\tilde{\Pi}_2(\mathscr{E})$.

\begin{proposition}
	Let $\mathscr{E}$ be a Hilbert $\mathfrak{A}$-module. Then $\left(\tilde{\Pi}_1(\mathscr{E}),\tilde{\pi}_1(\cdot)\right) $ is  a Banach space.
\end{proposition}
\begin{proof}
	Let $\left(T_n:n\in\mathbb{N}\right)$ be a Cauchy sequence in $\left(\tilde{\Pi}_1(\mathscr{E}),\tilde{\pi}_1(\cdot)\right) $.  It is a Cauchy sequence with respect to the operator norm; see Proposition \ref{ph}(2). There exists $ T\in\mathcal{L}(\mathscr{E})$ such that $\left\| T_n-T\right\|\rightarrow 0$.	Assume towards a contradiction that $\tilde{\pi}_1(T_n-T)\nrightarrow0$. Then equation \eqref{absol} ensures that there exist $\varepsilon>0$ and $n_0\in\mathbb{N}$ such that
	\begin{equation}\label{ep}
	\tilde{\pi}_1(T_{n_0}-T)> \varepsilon \quad\text{and}\quad \tilde{\pi}_1\left(T_{n_0}-T_{m}\right)<\frac{\varepsilon}{3}\quad\text{for each}\quad m\geq n_0.
	\end{equation}
	Then there  exist  $x_1,\ldots,x_k\in\mathscr{E}$  with $\mu_k(x_1,\ldots, x_k)\leq 1$ such that
	\begin{equation}\label{l22}
	\varepsilon<\left\| \sum_{i=1}^k\left\langle\left|T_{n_0}-T\right|x_i,x_i\right\rangle\right\|. 
	\end{equation}
	Pick $m_0\geq n_0$ such that $\left\|T_{m_0}-T\right\|<\varepsilon/3k$.
	 Lemma \ref{hh}  ensures that there exist $U,V\in\mathcal{U}(\mathscr{E})$ such that
	\begin{equation}\label{tt}
	\left|T_{n_0}-T\right|=U^*\left|T_{n_0}-T_{m_0} \right|U+  V^*\left|T_{m_0}-T \right|V+\frac{\varepsilon}{3k} 1_{\mathcal{L}(\mathscr{E})}.
	\end{equation}
	It follows from \eqref{ep}, \eqref{l22}, and \eqref{tt} that
	\begin{align*}
	\varepsilon<\left\|\sum_{i=1}^k\left\langle\left|T_{n_0}-T\right|x_i,x_i\right\rangle\right\|&\leq\left\|\sum_{i=1}^k\left\langle\left|T_{n_0}-T_{m_0}\right|Ux_i,Ux_i\right\rangle\right\|\\&+\left\|\sum_{i=1}^k\left\langle\left|T_{m_0}-T\right|Vx_i,Vx_i\right\rangle\right\|+\frac{\varepsilon}{3}<\varepsilon,
	\end{align*}
which is a contradiction. Thus $\tilde{\pi}_1(T_n-T)\rightarrow0$. Hence $\left(\tilde{\Pi}_1(\mathscr{E}),\tilde{\pi}_1(\cdot)\right) $ is  a Banach space.
\end{proof}  
 Stern and van Suijlekom \cite[Theorem 3.28]{sche} introduced a Schatten $p$-norm $\left\|\cdot\right\|_{\left[ p\right]}$ for $1\leq p$ such that $\left\|T\right\|_{\left[p\right]}=\sup_{\tau\in\Omega\left( \mathfrak{A}\right)}\left\|\phi_{\tau}(T)\right\| _{\left(p\right)}$ for $T\in\mathcal{L}(\mathscr{E})$, where $\mathscr{E}$ is a Hilbert $C^*$-module over a commutative $C^*$-algebra $\mathfrak{A}$. In this case, it is deduced from \cite[Theorems 3.5 and 3.28]{sche} and Theorem \ref{ffrr} that $\tilde{\pi}_p(T)=\left\|T\right\|_{\left[ p\right]}$ for $p\in\left\lbrace 1,2\right\rbrace$.  In what follows, we extend the definition of $\left\|\cdot\right\|_{\left[ p\right]}$ for an arbitrary Hilbert $C^*$-module $\mathscr{E}$. We define $\left\|T\right\|_{\left[ p\right]}=\sup_{\tau\in\mathcal{PS}\left( \mathfrak{A}\right)}\left\|\phi_{\tau}(T)\right\|_{\left(p\right)}$  for $T\in\mathcal{L}(\mathscr{E})$. It follows from \cite[Theorem 5.1.11]{mor} that $\left\|T\right\|=\sup_{\tau\in\mathcal{PS}\left( \mathfrak{A}\right)}\left\|\phi_{\tau}(T)\right\|$. Define $\mathcal{L}^p(\mathscr{E})=\left\lbrace T\in\mathcal{L}(\mathscr{E}):\left\|T\right\|_{\left[ p\right]}<\infty \right\rbrace $. It is easy to see that $\left\|T\right\|_{\left[ p\right]}$ is a norm on $\mathcal{L}^p(\mathscr{E})$ with expected properties. In the following example, we show that the equality $\tilde{\pi}_p(T)=\left\|T\right\|_{\left[ p\right]}$ is not true, in general.
\begin{example}
Suppose that $\mathscr{E}=\mathfrak{A}=\mathbb{M}_n(\mathbb{C})$. Let $\tau\in\mathcal{PS}\left(\mathfrak{A}\right) $. It follows from \cite[Example 5.1.1]{mor} that there exists a unit vector $f_1$ in the Hibert space $\mathscr{H}=\mathit{l}^2_n$ such that $\tau\left(S\right)=\left\langle Sf_1,f_1\right\rangle$ for each $S\in\mathbb{M}_n(\mathbb{C})$. The Hilbert space $\mathscr{H}$ has an orthonormal basis $\left\lbrace f_1,f_2,\ldots,f_n\right\rbrace $. Let $T\in\mathcal{L}\left(\mathscr{E}\right)=\mathbb{M}_n(\mathbb{C})$. Then $T$ is uniquely determined by the matrix  $[\left\langle Tf_j,f_i \right\rangle]_{i,j=1}^{n}$. Suppose  that $S\in\mathcal{N}_{\tau}$. We observe that  $\left\|Sf_1\right\| =\tau\left(S^*S \right)^{\frac{1}{2}}=0 $. Thus, $\mathcal{N}_{\tau}$ consists of all matrices in $\mathbb{M}_n(\mathbb{C})$ such that their first column is zero. Hence $\mathscr{H}_{\tau}=\mathit{l}^2_n$ and $\left\|\phi_{\tau}(T)\right\|_{\left(p\right)}=\left\|T\right\|_{\left( p\right)}$. Therefore, $\left\|T\right\|_{\left[ p\right]}=\sup_{\tau\in\mathcal{PS}\left( \mathfrak{A}\right)}
\left\|\phi_{\tau}(T)\right\|_{\left(p\right)}=\left\|T\right\|_{\left( p\right)}$. For example, if $p=1$, then $\left\|1_{\mathfrak{A}}\right\|_{\left[1\right]}=n$ while $\tilde{\pi}_1(1_{\mathfrak{A}})=1$.
\end{example}
In general, an operator in $\mathcal{L}(\mathscr{E})$ does not necessarily have  a polar decomposition. Weggo-Olsen  stated a necessary and sufficient condition for operators in $\mathcal{L}(\mathscr{E})$ to admit a polar decomposition; see also \cite{FS}.

\begin{lemma}\cite[Theorem 15.3.7]{wegga}\label{pd}
Suppose  $T\in\mathcal{L}(\mathscr{E},\mathscr{F})$. Then the following conditions are equivalent:
\begin{enumerate}
\item 
The operator $T$ has a unique polar decomposition $T=W\left|T\right|$, where $W\in\mathcal{L}(\mathscr{E},\mathscr{F})$ is a partial isometry for which $\rm{ker}(W)=\rm{ker}(T)=\rm{ker}(\left| T\right|)$ and $\rm{ker}(W^*)=\rm{ker}(T^*)=\rm{ker}(\left| T^*\right| )$.
\item 
$\mathscr{E}=\rm{ker}(T)\oplus\overline{\rm{ran}(T^*)}$ and $	\mathscr{F}=\rm{ker}(T^*)\oplus\overline{\rm{ran}(T)}$.	
\end{enumerate}
In this situation, $W^*W$ is the projection onto  $\overline{\rm{ran}(\left|T\right| )}=\overline{\rm{ran}(T^* )}$ and $WW^*$ is the projection onto  $\overline{\rm{ran}(\left|T^*\right| )}=\overline{\rm{ran}(T )}$.
\end{lemma}
\begin{lemma}\cite[Lemma 3.12]{pol}\label{pol}
Let $	T=W\left|T\right|$ be the polar decomposition of $T\in\mathcal{L}(\mathscr{E},\mathscr{F})$. Then for all $\alpha>0$, the following statements hold:
\begin{enumerate}
\item 
$W\left|T\right|^{\alpha}W^*=\left(W\left|T\right|W^*\right)^{\alpha} =\left|T^*\right|^{\alpha}$;
\item 
$W^*\left|T^*\right|^{\alpha}W=\left(W^*\left|T^*\right|W\right)^{\alpha} =\left|T\right|^{\alpha}$.
\end{enumerate}
\end{lemma}
The following proposition gives a  condition to achieve more properties for $\tilde{\Pi}_2(\mathscr{E},\mathscr{F})$ and $\tilde{\Pi}_1(\mathscr{E})$.
\begin{proposition}\label{ff}
Let $\mathscr{E}$ and $\mathscr{F}$ be Hilbert $\mathfrak{A}$-modules.
\begin{enumerate}
\item 
For $T\in\mathcal{L}(\mathscr{E},\mathscr{F})$, where $	\mathscr{E}=\rm{ker}(T)\oplus\overline{\rm{ran}(T^*)}$ and $	\mathscr{F}=\rm{ker}(T^*)\oplus\overline{\rm{ran}(T)}$, it holds that
\begin{equation}\label{67}
\tilde{\pi}_2(T)=\tilde{\pi}_2(T^*).
\end{equation} 
\item 
For $T,S\in\mathcal{L}(\mathscr{E})$, where $	\mathscr{E}=\rm{ker}(T)\oplus\overline{\rm{ran}(T^*)}=\rm{ker}(T^*)\oplus\overline{\rm{ran}(T)}$, it holds that
\begin{equation}\label{12}
\tilde{\pi}_1(T)=\tilde{\pi}_1(T^*)\quad\text{and}\quad\tilde{\pi}_1(TS)\leq\left\|S\right\| \tilde{\pi}_1(T). 
\end{equation} 
\end{enumerate}
\end{proposition}
\begin{proof}
(1) From Lemma \ref{pol}, the operator $T$ has a polar decomposition $T=W\left|T\right|$. In Lemma \ref{pd}, put $\alpha=2$. We observe that $W\left|T\right|^2W^*=\left|T^*\right|^2$ and $W^*\left|T^*\right|^2W=\left|T\right|^2$. It follows from \eqref{1} that $\tilde{\pi}_2(T)=\tilde{\pi}_2(T^*)$.

(2) From Lemma \ref{pol}, the operator $T$ has a polar decomposition $T=W\left|T\right|$. In Lemma \ref{pd}, put $\alpha=1$. We see that $W\left|T\right|W^*=\left|T^*\right|$ and $W^*\left|T^*\right|W=\left|T\right|$. It follows from \eqref{absol} that $\tilde{\pi}_1(T)=\tilde{\pi}_1(T^*)$.  We conclude from Proposition \ref{ph}(3) that $\tilde{\pi}_1(TS)=\tilde{\pi}_1(S^*T^*)\leq\left\|S\right\|\tilde{\pi}_1(T)$.
\end{proof}

\begin{remark}
	Inspired by \cite[Theorem 3.2]{lance} and \eqref{67}, if $T$ is a closed range operator in $\mathcal{L}(\mathscr{E},\mathscr{F})$, we have $\tilde{\pi}_2(T)=\tilde{\pi}_2(T^*)$. Furthermore,  for a closed range operator in $\mathcal{L}(\mathscr{E})$,   equation \eqref{12} is established.
\end{remark}
\begin{remark}
	In virtue of Proposition \ref{ff},	if  $\rm{ker}(T)$ is complemented for each $T\in\mathcal{L}(\mathscr{E},\mathscr{F})$, then $\tilde{\Pi}_2(\mathscr{E},\mathscr{F})$ is a Banach $*$-algebra. Furthermore, for $\mathcal{L}(\mathscr{E})$, we conclude that $\tilde{\Pi}_1(\mathscr{E})$ is a Banach $*$-algebra and a two-sided ideal of $\mathcal{L}(\mathscr{E})$. The complementedness of $\rm{ker}(T)$ is possible in the case where  $\mathscr{E}$ is a Hilbert $C^*$-module over a $C^*$-algebra of compact operators; see \cite[p. 2]{moslehian} and \cite{1, 7}.  
\end{remark}

\
\medskip

\noindent \textit{Acknowledgment.} The authors would like to thank the referee for carefully reading the paper.\\

\noindent \textit{Conflict of Interest Statement.} On behalf of all authors, the corresponding author states that there is no conflict of interest.\\

\noindent\textit{Data Availability Statement.} Data sharing not applicable to this article as no datasets were generated or analysed during the current study.
\medskip
\bibliographystyle{amsplain}

\end{document}